\documentclass[12pt]{amsart}
\usepackage{amsmath}
\usepackage{amssymb,amscd}
\usepackage{amscd}
\usepackage{amsmath,latexsym,amssymb,amsthm}
\usepackage{microtype}

\usepackage{tikz}
\usepackage{bbm}

\theoremstyle{plain}
\newtheorem{thm}[subsection]{Theorem}
\newtheorem{lem}[subsection]{Lemma}

\newcommand{\card}[1]{\ensuremath{ \# #1 }}

\newcommand{\set}[1]{\ensuremath{\left \{ #1 \right \} } }

\theoremstyle{definition}

\newtheorem{defn}[subsection]{Definition}

\def\v{\vee}
\def\w{\wedge}

\def\L{\mathcal{L}}

\def\a{\alpha}

\def\b{\beta}

\def\s{\sigma}

\def\d{\delta}
\def\g{\gamma}

\def\v{\vee}
\def\w{\wedge}

\def\L{\mathcal{L}}

\def\a{\alpha}

\def\b{\beta}

\def\s{\sigma}

\def\d{\delta}
\def\g{\gamma}

 \title{An identity of distributive lattices}

\author{Himadri Mukherjee}
\address{Department of Mathematics and Statistics, IISER-Kolkata}
\email{himadri@iiserkol.ac.in}
\date{\today}

\begin{document}

\subjclass[2010]{Primary 06D50}

\begin{abstract}
In a finite distributive lattice $\L$ we define two functions $s(\a)=\card{\{\d \in \L | \d \leq \a\}}$ and $l(\a)=\card{\{\d \in \L | \d \geq \a\}}$. In this present article we prove that the sum of these two functions over a finite distributive lattice are equal.
\end{abstract}

\maketitle

\section{Introduction}

In this present article we look at the functions $s(\a)=\card{\{\d \in \L | \d \leq \a\}}$ which we will call ``smaller'' function and $l(\a)=\card{\{\d \in \L | \d \geq \a\}}$ which we will call ``larger'' function, these functions are closely related to the rank $r(\a)$ and co-rank $cr(\a)$ functions of the lattice and clearly not same unless the lattice in question is a chain lattice. Even though it's clear that $\sum_{\d \in \L} r(\d)=\sum_{\d \in \L} cr(\d)$ it is not trivial to demonstrate that ``lower sum'' $\sum_{\d \in \L} s(\d)$ is equal to the ``upper sum''  $\sum_{\d \in \L} l(\d)$. In this present article we provide an interesting proof of the fact.
We give an example as below: Hasse diagram of a distributive lattice $\L$ is given below, the nodes in the first graph denotes the numbers $s(\d)$ and the numbers in the nodes of the second graph denotes $l(\d)$. Even though the numbers are different in places the sum comes to 41 in both the cases.

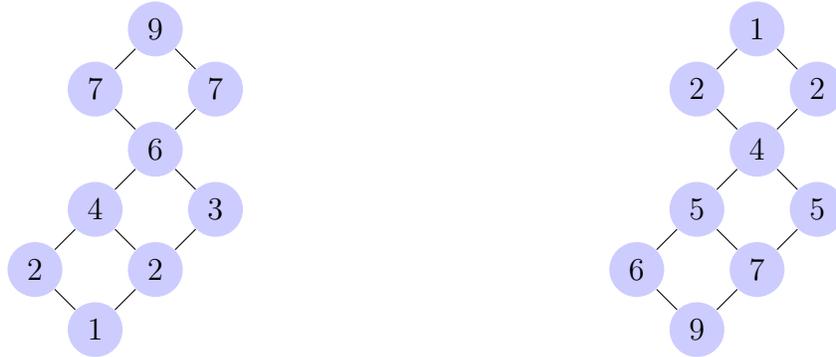
\begin{center}
\begin{tikzpicture}
  [scale=.8,auto=left,every node/.style={circle,fill=blue!20}]
  \node (n1) at (0,0) {1};
  \node (n2) at (-1,1)  {2};
  \node (n3) at (1,1)  {2};
  \node (n4) at (0,2) {4};
  \node (n5) at (2,2)  {3};
  \node (n6) at (1,3)  {6};
  \node (n7) at (0,4) {7};
  \node (n8) at (2,4) {7};
  \node (n9) at (1,5) {9};

  \node (m1) at (10,0) {9};
  \node (m2) at (9,1)  {6};
  \node (m3) at (11,1)  {7};
  \node (m4) at (10,2) {5};
  \node (m5) at (12,2)  {5};
  \node (m6) at (11,3)  {4};
  \node (m7) at (10,4) {2};
  \node (m8) at (12,4) {2};
  \node (m9) at (11,5) {1};

  \foreach \from/\to in {n1/n2,n2/n4,n4/n6,n6/n7,n7/n9,n9/n8,n8/n6,n6/n5,n5/n3,n3/n1,n3/n4}
    \draw (\from) -- (\to);

  \foreach \from/\to in {m1/m2,m2/m4,m4/m6,m6/m7,m7/m9,m9/m8,m8/m6,m6/m5,m5/m3,m3/m1,m3/m4}
    \draw (\from) -- (\to);

\end{tikzpicture}

\end{center}

Let us state below our main theorem that the lower sum and upper sum of a finite distributive lattice are equal.
\begin{thm}\label{main} For a finite distributive lattice $\L$ we have, $\sum_{\d \in \L} s(\d)=\sum_{\d \in \L} l(\d)$
\end{thm}

We take a maximal join irreducible element $\b \in J$ and look at the pruned lattice, see \ref{prune} for a definition, $\L_\b$. We then apply induction on the number of join irreducibles of the lattice $\L$. For a trivial lattice with $\card{J(\L)}=1$ the claim is obvious. Before the proof we need to recall few definitions and prove a few lemmas.

\section{Definitions and Lemmas}

\begin{lem}\label{join_lemma} Let $\theta > \delta $ be a cover in a distributive lattice $\L$, then there is a unique join irreducible $\epsilon$ such that $\theta = \delta \v \epsilon$.
\end{lem}
\begin{proof}
Using Birkhoff's representation theorem \cite{GG} we write the elements $\theta, \delta \in \L$ as ideals in the poset of join irreducibles $I_\theta$, $I_\delta$. Since $\theta $ covers $\delta$ we have $I_\theta \setminus I_\delta =\{ \epsilon \}$ for a join irreducible element $\epsilon$. It is clear that this epsilon serves the purpose.
\end{proof}

\begin{lem}\label{join_lemma_2} If $\theta \v \delta \geq \b$ where $\b$ is a join irreducible then either $\theta$ or $\delta $ is larger than $\b$
\end{lem}
\begin{proof}
Since $\theta \v \beta \geq \d $ we have $I_\beta \subset I_\theta \cup I_\delta $. But we have $\beta \in I_\beta$, $\beta$ being a join irreducible. So we have $\beta \in I_\theta$ or $\beta \in I_\delta$.
\end{proof}


For a lattice $\L$ and an elements $\d \in \L$ let us define a sublattice $\Gamma_{\d}=\{\b \in \L : \b \sim \d\}$. Let us also define a set of functions $f,s,l$ on the lattice $\L$ taking values in $\mathbb{N}$ as follows:
$s(\d)= \card{ \{ \b : \b \leq \d \}}$ and $l(\d)=\card { \{\b : \b \geq \d \}}$ and $f(\d)=\card{\Gamma_{\d}}=s(\d)+l(\d)-1$.


Let us recall from \cite{noncomp} the following theorem, where the number of non-comparable pairs of elements in a distributive lattice $\mathcal{L}$ was denoted by $n(\L)$.
\begin{thm}\label{first}\label{equation}
For a distributive lattice $\L$ we have \[n(\mathcal{L})=1/2(\card{\mathcal{L}}^2 - \sum_{\d \in \L} f(\d))\]
\end{thm}
In light of the above theorem we would like to calculate the sum $\sum_{\d \in \L} f(\d)$. Using the functions $s$ and $l$ we can rewrite the sum as $\sum_{\d \in \L} s(\d) + \sum_{\d \in \L} l(\d) -\card{\L}$. So to understand the sum $F=\sum_{\d \in \L} f(\d)$ we have to understand the sums $S=\sum_{\d \in \L} s(\d)$ and $L=\sum_{\d \in \L} l(\d)$. The following theorem is one step towards understanding the two sums $S$ and $L$.

Let us start with a recollection of a concept called pruning, which will be essential for the induction step, from the paper \cite{HM}. Let $\a \in \L$ be a maximal join-irreducible element. Let us define a sublattice $\L_\a$ of $\L$ which we called the pruned lattice of $\L$ with respect to the maximal join irreducible $\a$.

\begin{defn} \label{prune} The subset $\L_\a= \{ \b \in \L : \b \ngeq \a \}$ is called the pruned sublattice of the lattice $\L$ with respect to the maximal join irreducible $\a$ . So the subset $\L_\a$ consist of all the elements $\b$ which are either smaller than or non-comparable to $\a$.
\end{defn}
Let us also recall the following definition of embedded sublattices from \cite{HL}.
\begin{defn}\label{emb}A sublattice $\mathcal{D} \subset \L$ is called an embedded sublattice if for every $\theta, \delta \in \mathcal{D}$ whenever there are elements $\a , \b \in \L$ such that $\a \v \b =\theta$ and $\a \w \b =\delta$ then $\a , \b \in \mathcal{D}$.
\end{defn}

In the lemma below we will establish that ``pruned sublattice'' is truly a sublattice and is also embedded.
\begin{lem}\label{prune1} The pruned lattice $\L_\a$ is a sublattice of $\L$ that is embedded.
\end{lem}
\begin{proof} Pruned subset is a sublattice since if $x,y \in \L_\a$ then $x \w y \ngeq \a$ and if $x \v y \geq \a$ then by \ref{join_lemma_2} we have either $x \geq \a $ or $y \geq \a$ which leads to contradiction. For the embedded part if for $\theta, \delta $ we have $\theta \v \delta$ and $\theta \w \delta $ in $\L_\a$ then neither $x $ nor $y$ can be larger than $\a$ again by \ref{join_lemma_2}
\end{proof}

With the above notation in mind let us define a new set of functions, let $\a$ be a maximal join irreducible and let $\L_\a$ be the corresponding pruned sublattice. Then let us define $l_\a(\d)= \card{\set{x \in \L | x \geq \d} \cap \L_\a }$, $s_\a(\d)=\card{\set{x \in \L| x \leq \d} \cap \L_\a }$, $f_\a(\d)=\card{\Gamma_\d \cap \L_\a}$.

\begin{lem} $l(\d)=l_\a(\d)+l(\a \cap \d)$
\end{lem}
\begin{proof} From the definitions we have $l(\d)-l_\a(\d)=\card{ \{ \b : \b \geq \a \; \mbox{and } \; \b \geq \d\}}$. But $\b \geq \a$ and $\b \geq \d $ if and only if we have $\b \geq \a \v \d $. So putting all these together we get $l(\d) - l_\a(\d) + \card{\{\b : \b \geq \a \v \d \}}$. From the definition of the function $l$ we write $\card{\{\b : \b \geq \a \v \d \}}$ as $l(\a \v \d)$ and get the statement of the lemma.
\end{proof}

\begin{lem} $s(\d)=s_\a(\d)+\card{\{\b : \b \geq \a \; \mathrm{and } \; \b \leq \d \}}$. To simplify the notations let us denote $\card{\{\b : \b \geq \a \; \mathrm{and } \; \b \leq \d \}} $ by $h_\a(\d)$
\end{lem}
\begin{proof} Proof of this fact is straightforward from the definition.
\end{proof}

We want to show that the sum $I= \sum_{\d \in \L} l(\d) -s(\d)=0$. To show that we want to use an induction argument on the size of the lattice $\L$. Let $\a$ be a maximal join irreducible in the lattice $\L$ and $\L_\a$ be its pruned lattice, let us denote the complement of the pruned lattice by $X_\a = \L \setminus \L_\a$. Note that from the definition of the pruned lattice the structure of the set $X_\a$ can be written rather easily so we include that as a lemma.
\begin{lem} $X_\a=\{ \b \in \L : \b \geq \a \}$, is an embedded sub-lattice.
\end{lem}
\begin{proof} Clear from the definition \ref{prune}.
\end{proof}

With the above definitions in mind we will split the sum $I= \sum_{\d \in \L} (l(\d) -s(\d))$ into two sums such as $I= \sum_{\d \in \L_\a} (l(\d)-\s(\d))+\sum_{\d \in X_\a}(l(\d)-s(\d))$. For notational simplicity let us denote the first part of the sum namely $\sum_{\d \in \L_\a} (l(\d)-\s(\d))$ as $S_\a$ and the second part namely $\sum_{\d \in X_\a}(l(\d)-s(\d))$ as $T_\a$. Thus we have $I=S_\a+T_\a$, and observe that if we prove $I=0$ then the main theorem \ref{main} follows.

\begin{lem} \[S_\a=\sum_{\d \in \L_\a} l(\a \v \d)\]
\end{lem}
\begin{proof} \begin{eqnarray}
S_\a &=& \sum_{\d \in \L_\a } (l(\d)-s(\d)) \nonumber \\
&=&\sum_{\d \in \L_\a} l_\a(\d) + l(\a \v \d) - s_\a(\d) - h_\a(\d) \nonumber \\
\end{eqnarray}
But since $\d \in \L_\a$ we do not have any $\b \in \L$ which is larger than $\a$ and smaller than $\d$, since that will imply $\d \geq \a$ contrary to our assumption. So the number $h_\a(\d)=\card{\{\b : \b \geq \a \; \mathrm{and } \; \b \leq \d \}}$ is zero. Thus :
\[S_\a= \sum_{\d \in \L_\a} l_\a(\d) + l(\a \v \d) - s_\a(\d)\]

But by induction the sum $\sum_{\d \in \L_\a}( l_\a(\d) - s_\a(\d))$ is zero. So we have $S_\a=\sum_{\d \in \L_\a} l (\a \v \d )$ as required.
\end{proof}

Let us now similarly investigate the number $T_\a$.
\begin{lem} $T_\a = -\sum_{\d \in X_\a} s_\a(\d)$
\end{lem}
\begin{proof} \begin{eqnarray} T_\a &=& \sum_{\d \in X_\a} (l(\d) -s(\d)) \nonumber \\
&=& \sum_{\d \in X_\a } (l_\a(\d)+l(\a \w \d) - s_\a(\d)-h_\a(\d)\nonumber \nonumber \\
\end{eqnarray}

Since $\d \geq \a $ the number $l_\a(\d)=0$ and $\a \w \d = \d$. So we rewrite the above equation to:
\begin{eqnarray}
& = &\sum_{\d \in X_\a} (l(d)-h_\a(\d)) - \sum_{\d \in X_\a} s_\a(\d)\nonumber \\
\end{eqnarray}
Looking at the sublattice $X_\a$ observe that the function $h_\a(\d)$ is just the ``smaller function'' for the lattice $X_\a$. Hence using the induction hypothesis we get \[\sum_{\d \in X_\a} (l(d)-h_\a(\d))=0\] Thus :
\[T_\a=-\sum_{\d \in X_\a} s_\a(\d)\]

\end{proof}

Let us summarise the above two lemmas in the following lemma:
\begin{lem} \[I=S_\a+T_\a=\sum_{\d \in \L_\a} l(\a \v \d) - \sum_{\d \in X_\a} s_\a(\d)\]
\end{lem}

\begin{defn} For $\g \in X_\a$ let us denote $C_\g = \{ x \in \mathcal{L}_{\a} | x \v \a = \g \}$
\end{defn}

\begin{lem} For any $\g \in X_\a$, $C_\g$ is non-empty and is a sublattice of $\L$.
\end{lem}
\begin{proof}
Let $a$ be the maximal element in $\L$ which is smaller than $\g$ but not larger than $\a$. If there exist no such element then every element $a$ smaller than $\g$ is also larger than $\a$, which will mean $\a$ is the smallest element in $\L$. But that is not the case since $\a$ is a maximal join irreducible. Consider $a \v \a$ , since both $a$ and $\a$ are smaller than $\g$, $a \v \a \leq \g$. If $\g = a \v \a $ then we are done, and if $\g$ covers $a \v \a $ then by \ref{join_lemma} we have a join irreducible $\epsilon$ such that $\g = \epsilon \v a \v \a$ or $\g = (\epsilon \v a) \v \a $ which shows $\epsilon \v a \in C_\g$. And if there is $\g_1$ such that $\g > \g_1 > a \v \a$, by induction on the rank of $\g$ we have an element $b$ such that $\g_1 = b \v \a$. But note that $\g > b$ and $b \ngeq \a$ and we have $\g > a \v b$ and $a \v b \ngeq \a$ which contradicts the maximality of $a$.

For the lattice part, if $x, y \in C_\g$ then $x \v y \v \a = x \v \g $ but since $x \in C_\g$ we have $x \leq \g$ so $x \v \g = \g$. Similarly for $(x \w y) \v \a= (x \v \a) \w (y \v \a)= \g \w \g= \g$.
\end{proof}

\begin{defn} Let $x_\g$ be the maximal element of $C_\g$.
\end{defn}

\begin{lem}\label{one} For each $\d \in X_\a$ we have $s_\a(\d)=s_\a(x_\d)=s(x_\d)$
\end{lem}
\begin{proof} Since $\d \in X_\a$ we have $\d = x_\d \v \a$ from the definition of $x_\d$. And we know $s_\a(\d)=\card{ \{ \b \in \L_\a : \b \leq \d \} }$. But note by the following argument that the set $\{ \b \in \L_\a : \b \leq \d\}$ is equal to the set $\{\b \in \L_\a : \b \leq x_\d \}$.

Clearly if $\b \leq x_\d$ then $\b \leq \d = x_\d \v \a$ for the inclusion other way around see that if $y \in \L_\a $ and $y \leq \d$ then if $y$ is non-comparable to $x_\d$ then we have $x_\d \leq y \v x_\d \leq \d$ which contradicts the maximality of $x_\d$ unless $y \v x_\d =\d$. But $\d \in X_\a$ which means $\d \geq \a$ or $y \v x_\d \geq \a $ so by \ref{join_lemma_2} we have either $y$ or $x_\d$ larger than $\a$ which is a contradiction to the assumption that both $y$ and $x_\d$ are in $\L_\a$. So it means $y$ is comparable to $x_\d$ but it cannot be larger since it would violate the maximality of $x_\d$ so we have $y \leq x_\d$, or the two sets are equal.

Which means $s_\a(\d)=\card{\set{ \b \in \L_\a : \b \leq \d}} = \card{\set{\b \in \L_\a : \b \leq x_\d}}=s_\a(x_\d)$. Now since $x_\d$ is by definition in $\L_\a$ we have $s_\a(x_\d)=s(x_\d)$.

\end{proof}

\begin{lem}\label{two} \[\sum_{\d \in \L_\a} l(\a \v \d)=\sum_{\g \in X_\a} \card{C_\g}l(\g)\]
\end{lem}
\begin{proof}If $\d \in C_\g$ then by definition we have $\g = \a \v \d$ or $l(\a \v \d)= l(\g)$. So $\sum_{\d \in \L_\a} l(\a \v \d)= \sum_{\g = \a \v \d \in X_\a} l(\g)$ we can do this change of variable since we have $\cup_{\g \in X_\a} C_\g = \L_\a$. But $\sum_{\g = \a \v \d \in X_\a} l(\g)=\sum_{\g \in X_\a} \card{C_\g}l(\g)$

\end{proof}

\begin{lem}\label{three} \[\sum_{\g \in X_\a} \card{C_\g} l(\g)=\sum_{\g \in X_\a}s(x_\g)\]
\end{lem}
\begin{proof} Let $A_\g$ be the set $\set{ \b \in \L : \b \leq x_\g}$. So $s(x_\g)=\card{A_\g}$. On the set $A_\g$ let us give a relation $\sim$ as follows: $a,b \in A_\g$ we will call $a \sim b $ if and only if $a \v \a = b \v \a$. Observe that this is an equivalence relation. Also observe that the equivalence classes are $C_{\g_i}=\set{a | a \v \a =\g_i}$ where $\g_i \leq \g$. Thus we have $s(x_\g)=\card{A_\g}= \sum_{\g_i \in X_\a, \g_i \leq \g} C_{\g_i}$, if we sum the both sides up we get \[\sum_{\g \in X_\a} s(x_\g)=\sum_{\g \in X_\a} \sum_{\g_i \leq \g, \g_i \in X_a } C_{g_i}\] \[= \sum_{\g \in X_\a} \card{C_\g} \sum_{y \geq \g} 1 \] Since each $C_{\g_i}$ is appearing $\sum_{y \geq \g_i} 1$ many times in the right hand sum. Now observe that $\sum_{y \geq \g} 1=l(\g)$ so we can finally write: \[ \sum_{\g \in X_\a} s(x_\g)=\sum_{\g \in X_\a} \card{C_\g} l(\g)\].
\end{proof}

\begin{thm}\label{main two} $I=S_\a+T_\a=0$
\end{thm}
\begin{proof} \[S_\a+T_\a=\sum_{\d \in \L_\a}l(\a \v \d) - \sum_{\a \in X_\a} s_\a(\d)\]
\[=\sum_{\g \in X_\a} \card{C_\g} l(\g) - \sum_{\d \in X_\a } s_\a(x_\d)\]
by \ref{one} and \ref{two}
\[=\sum_{\g \in X_\a} \card{C_\g}l(\g)-s(x_\g)=0 \]
by \ref{two} and after changing the running variable to a common variable. And the last equality follows from \ref{three}. So we have proved the main theorem \ref{main}.

\end{proof}

\bibliographystyle{abbrv}
\bibliography{main}

\end{document}